\documentclass{article}
\usepackage{graphicx}
\usepackage{amsthm, amsmath, amssymb, tikz,hyperref}
\usepackage{enumitem}
\usepackage{xifthen}
\usepackage{todonotes}
\usepackage[margin=1.2in]{geometry} 

\usetikzlibrary{calc,shapes, backgrounds}

\newtheorem{lemma}{Lemma}
\newtheorem{theorem}{Theorem}

\newtheorem{definition}{Definition}

\newcommand{\ds}{\displaystyle}

\newcommand{\lp}{\left (}
\newcommand{\rp}{\right )}

\pgfdeclarelayer{edgelayer}
\pgfdeclarelayer{nodelayer}
\pgfsetlayers{edgelayer,nodelayer,main}

\begin{document}

\title{Erd\H os-Szekeres theorem for cyclic permutations}
\author{
\'Eva Czabarka
\thanks{University of South Carolina, Columbia, SC 29208,
({\tt czabarka@math.sc.edu}), and Visiting Professor, University of Johannesburg, South Africa.} \and
Zhiyu Wang \thanks{University of South Carolina, Columbia, SC 29208,
({\tt zhiyuw@math.sc.edu}). The research of this author was supported in part by the NSF-DMS grants \#1604458, \#1604773, \#1604697
and \#1603823, ``Collaborative Research: Rocky Mountain - Great Plains Graduate Research
Workshops in Combinatorics." as well as NSF-DMS grant \#1600811.} 
}

\maketitle

\begin{abstract}
We provide a cyclic permutation analogue of the Erd\H os-Szekeres theorem. In particular, we show that every cyclic permutation of length $(k-1)(\ell-1)+2$ has either an increasing cyclic sub-permutation of length $k+1$ or a decreasing cyclic sub-permutation of length $\ell+1$, and show that the result is tight. We also characterize all maximum-length cyclic permutations that do not have an increasing cyclic sub-permutation of length $k+1$ or a decreasing cyclic sub-permutation of length $\ell+1$. 
\end{abstract}
 
\section{Introduction}

The study of the longest monotone subsequence of a finite sequence of numbers has inspired a body of research in mathematics, bioinformatics, and computer science. In 1935, Erd\H os and Szekeres  \cite{Erdos}  showed in their namesake theorem that any permutation of $\{1,2,...,k\ell+1\}$ has an increasing subsequence of length $k+1$ or a decreasing subsequence of length $\ell+1$. As a sequence $(a_1,\ldots,a_n)$ can be represented by a set of $n$ points of the form $(i,a_i)$ in the plane,
the Erd\H os-Szekeres theorem can be interpreted geometrically in the following way:  for any set of $k\ell+1$ points in the plane, no two of which are on the
same horizontal or vertical line, there exists a polygonal path of either $k$ positive-slope edges or $\ell$ negative-slope edges. It follows immediately from the Erd\H os-Szekeres theorem that the expected length of a longest increasing subsequence (LIS) in a random permutation of length $n$ is at least $\frac{1}{2}\sqrt{n}$. Moreover, computing LIS is also used in MUMmer systems for aligning whole genomes \cite{genomes}. A natural extension of the well-known Erd\H os-Szekeres theorem is to consider its analogue to cyclic sub-permutations. 

\begin{definition} A cyclic sub-permutation $\tau$ of a cyclic permutation $\sigma$ is the restriction of $\sigma$ on $\tau$, i.e. remove all elements not in $\tau$ from $\sigma$. \end{definition}

For example, $(135)$ is a cyclic sub-permutation of the cyclic permutation $(12345)$. 

\begin{definition} A cyclic permutation is increasing if it can be written in the form $(j_1,j_2,\ldots,j_n)$ with $j_1<j_2<\ldots< j_n$. Similarly, a cyclic permutation is decreasing if it can be written in the form $(j_1,j_2,\ldots,j_n)$ with $j_1>j_2>\ldots> j_n$.
	
\end{definition}

For example, $(6,1,4,2, 7,3,5)$ is a cyclic permutation whose longest increasing cyclic sub-permutation is $(1,2,3,5,6)$ and whose longest decreasing cyclic sub-permutations are $(7,5,4,2)$ or $(7,6,4,2)$.

Cyclic permutations can be viewed as circular lists, which arise naturally in the field of phylogenetics since the genomes of bacteria are considered to be circular. Geometrically, an increasing/decreasing cyclic subsequence of a circular list corresponds to a polygonal path of positive/negative-slope edges when the points are drawn on the side of a cylinder. Albert et al. in \cite{LIS} give a Monte Carlo algorithm to compute the longest increasing circular subsequence with worst case run-time $O(n^{3/2} \log n)$ and also showed that the expected length $\mu(n)$ of the longest increasing circular subsequence satisfies $\ds\lim_{n\to \infty} \frac{\mu(n)}{2\sqrt{n}} = 1$. We extend the Erd\H os-Szekeres theorem to cyclic permutations and examine the structures of the extremal constructions achieving the lower bound for our theorem.

\begin{definition}
	Given positive integers $k$ and $\ell$, let $\alpha(k,\ell)$ be the smallest positive integer $n$, such that for any cyclic permutation of length $n$, there exists either an increasing cyclic sub-permutation of length $k+1$, or a decreasing cyclic sub-permutation of length $\ell+1$.
\end{definition}

We show in Section~\ref{sc:maintheorem} that
\begin{theorem}\label{th:main} For $k, \ell\geq 1$,
\[ \alpha(k,\ell) =  (k-1)(\ell-1)+2.\]
\end{theorem}

\begin{definition}
	Given positive integers $k$ and $\ell$, let $\mathbb{C}_{k,\ell}$ be the set of cyclic permutations of length
	$(k-1)(\ell-1)+1$ that contain no increasing cyclic sub-permutation of length $k+1$, or decreasing cyclic sub-permutation of length $\ell+1$;
	let $\mathbb{S}_{k,\ell}$ be the set of linear permutations of length
	$k\ell$ that contain no increasing linear sub-permutation of length $k+1$, or decreasing linear sub-permutation of length $\ell+1$; and 
	let $\mathbb{Y}_{\ell,k}$ be the
	set of standard Young tableaux on a $\ell\times k$ rectangular diagram, i.e. the set of $\ell\times k$ matrices where the set of entries is $\{1,2,\ldots,k\ell\}$ 
	and each row and column forms an increasing sequence.
	\end{definition}
	
It was observed by Knuth [\cite{Knuth}, Exercise 5.1.4.9] (see also [\cite{Stanley}, Example 7.23.19(b)]) that the permutations in $\mathbb{S}_{k,\ell}$ are in bijection with $\mathbb{Y}_{\ell,k} \times \mathbb{Y}_{\ell,k}$ via the Robinson-Schensted correspondence. The hook-length formula \cite{hooklength} expresses the number of standard Young tableaux and allows us to directly compute $\vert\mathbb{S}_{k,\ell}\vert$, which increases rapidly as $k,\ell$ increase (see sequence  A060854 in the On-Line Encyclopedia of Integer Sequences). In particular, WLOG, assuming $k \leq l$ (since $|\mathbb{S}_{k,\ell}| = |\mathbb{S}_{\ell,k}|$), we have that

$$|\mathbb{S}_{k,\ell}| = \lp \frac{(\ell k)!}{1^1 2^2 \ldots k^k (k+1)^k \ldots \ell^k (\ell+1)^{k-1} \ldots (k+\ell-1) }\rp^2.$$

Although the Robinson-Schensted correspondence establishes the bijection between $\mathbb{S}_{k,l}$ and $\mathbb{Y}_{\ell,k} \times \mathbb{Y}_{\ell,k}$,
it is an algorithmic procedure which can be difficult to analyze.
Romik, in \cite{romik}, gave a simple description of the mapping from pairs of square Young Tableaux to elements of $\mathbb{S}_{k,k}$. Before we state the theorem, let us introduce a few definitions.

\begin{definition}\label{dfn:grid-function} The {\em grid-function} of an $\vec{a}=[a_1,\ldots,a_{k\ell}]\in\mathbb{S}_{k,\ell}$ is $\gamma_{\vec{a}}:[k\ell]\rightarrow[\ell]\times[k]$, defined 
	by $\gamma_{\vec{a}}(t)=(i,j)$ where $i$ is the length of the longest decreasing subsequence of $\vec{a}$ ending at $a_t$ and
	$j$ is the length of the longest increasing subsequence of $\vec{a}$ ending at $a_t$.
\end{definition}

\begin{definition}\label{dfn:grid-pair}
	The {\em grid-ranking} $R_{\vec{a}}=(r_{ij})$ and
	{\em grid-valuation} $V_{\vec{i,j}}=(v_{ij})$ are
	$\ell \times k$ matrices defined by $r_{ij}=\gamma^{-1}_{\vec{a}}(i,j)$, and $v_{ij}=a_{\gamma^{-1}(\ell+1-i,j)}$. 	
\end{definition}

Note that the Erd\H{o}s-Szekeres theorem implies that for a linear permutation $\vec{a}\in\mathbb{S}_{k,\ell}$,
the longest increasing subsequence has length $k$ and the longest decreasing subsequence has length $\ell$ (as both $k(\ell-1)+1$ and $(k-1)\ell+1$ are at most $ k\ell$), 
which means that $\gamma_{\vec{a}}$ indeed defines a function. 

Working towards our characterization of $\mathbb{C}_{k,\ell}$, Section~\ref{sc:linear} reproves the following result of \cite{romik}, partially for the sake of self-containment and partially for its use in thse proof of Theorem \ref{th:lowerbound}.
\begin{theorem}\label{th:linear} For positive integers $k,\ell$, $\mathbb{S}_{k,\ell}$ is isomorphic to $\mathbb{Y}_{\ell,k}\times\mathbb{Y}_{\ell,k}$. In particular, $\phi: \mathbb{S}_{k,\ell} \to \mathbb{Y}_{\ell,k}\times\mathbb{Y}_{\ell,k}$ defined by $\phi (\vec{a}) = (R_{\vec{a}}, V_{\vec{a}})$ is a bijection. 
\end{theorem}

In contrast to the exponential size of $\mathbb{S}_{k,l}$, $\mathbb{C}_{k,l}$ has at most $2$ elements and we can characterize them precisely. In particular, 
in Section~\ref{sc:lowerbound}, we show the following theorem:
\begin{theorem}\label{th:lowerbound} For $k, \ell\geq 1$, let $\mathbb{C}_{k,\ell}$ denote the set of cyclic permutations of $[(k-1)(\ell-1)+1]$ that contain no increasing cyclic sub-permutation of length $k+1$, or decreasing cyclic sub-permutation of length $\ell+1$. Then we have:
\begin{enumerate}[label=\rm{(\arabic*)}]
\item\label{atmost2} If $\min(k,\ell)\le 2$ then $\vert\mathbb{C}_{k,\ell}\vert=1$ and the single element of $\mathbb{C}_{k,\ell}$ is the decreasing cyclic permutation when $k\le 2$ and
the increasing cyclic permutation when $k\ge 3$.
\item\label{atleast3} If $\min(k,\ell)\ge 3$ then $\vert\mathbb{C}_{k,\ell}\vert=2$, and $(1,a_1,\ldots,a_{(k-1)(\ell-1)})\in \mathbb{C}_{k,\ell}$ precisely when the sequence it satisfies one of the following:
\begin{enumerate}[label=\rm{(\roman*)}]
\item\label{struct1} For each $(i,j)\in[\ell-1]\times[k-1]$,  $$a_{(j-1)(\ell-1)+i}=(\ell-1-i)(k-1)+j+1.$$
\item\label{struct2} For each $(i,j)\in[\ell-1]\times[k-1]$,
$$a_{(i-1)(k-1)+j}=(j-1)(\ell-1)+(\ell-i)+1.$$
\end{enumerate}
\end{enumerate}
\end{theorem}
Note that when $\min(k,\ell)=2$, the structures described in parts~\ref{atleast3}~\ref{struct1} and~\ref{struct2} are the same and coincide with the single structure described in part~\ref{atmost2}. Figure~\ref{fig:extremal example} illustrates the structures in parts~\ref{atleast3}~\ref{struct1} and~\ref{struct2} for $k=4$ and $\ell=5$.
The two extremal examples are $(1,11,8,5,2,12,9,6,3,13,10,7,4)$ and $(1,5,9,13,4,8,12,3,7,11,2,6,10)$ respectively.

	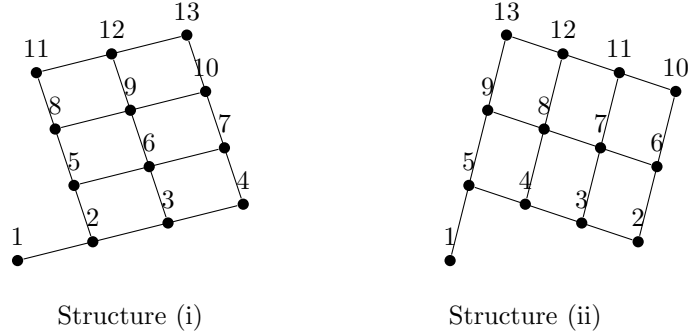
\begin{figure}[hbt]
		\begin{center}
			\tikzstyle{vertex}=[circle,fill=black,inner sep=1.5pt]
			\tikzstyle{straight edge}=[]
			
			\begin{tikzpicture}
			\begin{pgfonlayer}{nodelayer}
			\node [style=vertex,label=above:{$1$}] (0) at (0.25, 0.25) {};
			\node [style=vertex,label=above:{$2$}] (1) at (1.25, 0.5) {};
			\node [style=vertex,label=above:{$3$}] (2) at (2.25, 0.75) {};
			\node [style=vertex,label=above:{$4$}] (3) at (3.25, 1) {};
			\node [style=vertex,label=above:{$11$}] (4) at (0.5, 2.75) {};
			\node [style=vertex,label=above:{$12$}] (5) at (1.5, 3) {};
			\node [style=vertex,label=above:{$13$}] (6) at (2.5, 3.25) {};
			\node [style=vertex,label=above:{$8$}] (7) at (0.75, 2) {};
			\node [style=vertex,label=above:{$9$}] (8) at (1.75, 2.25) {};
			\node [style=vertex,label=above:{$10$}] (9) at (2.75, 2.5) {};
			\node [style=vertex,label=above:{$5$}] (10) at (1, 1.25) {};
			\node [style=vertex,label=above:{$6$}] (11) at (2, 1.5) {};
			\node [style=vertex,label=above:{$7$}] (12) at (3, 1.75) {};
			\node at (1.75,-.5) {Structure~\ref{struct1}};
			\node [style=vertex,label=above:{$1$}] (A1) at (6, 0.25) {};
			\node [style=vertex,label=above:{$2$}] (A2) at (8.5, 0.5) {};
			\node [style=vertex,label=above:{$3$}] (A3) at (7.75, 0.75) {};
			\node [style=vertex,label=above:{$4$}] (A4) at (7, 1) {};
			\node [style=vertex,label=above:{$5$}] (A5) at (6.25, 1.25) {};
			\node [style=vertex,label=above:{$6$}] (A6) at (8.75, 1.5) {};
			\node [style=vertex,label=above:{$7$}] (A7) at (8, 1.75) {};
			\node [style=vertex,label=above:{$8$}] (A8) at (7.25, 2) {};
			\node [style=vertex,label=above:{$9$}] (A9) at (6.5, 2.25) {};
			\node [style=vertex,label=above:{$10$}] (A10) at (9, 2.5) {};
			\node [style=vertex,label=above:{$11$}] (A11) at (8.25, 2.75) {};
			\node [style=vertex,label=above:{$12$}] (A12) at (7.5, 3) {};
			\node [style=vertex,label=above:{$13$}] (A13) at (6.75, 3.25) {};	
			\node at (7,-.5) {Structure~\ref{struct2}};
			
			\end{pgfonlayer}
			\begin{pgfonlayer}{edgelayer}
			\draw [style=straight edge] (0) to (3);
			\draw [style=straight edge] (10) to (12);
			\draw [style=straight edge] (7) to (9);
			\draw [style=straight edge] (4) to (6);
			\draw [style=straight edge] (4) to (1);
			\draw [style=straight edge] (5) to (2);
			\draw [style=straight edge] (6) to (3);
			\draw [style=straight edge] (A1) to (A13);
			\draw [style=straight edge] (A2) to (A10);
			\draw [style=straight edge] (A3) to (A11);
			\draw [style=straight edge] (A4) to (A12);
			\draw [style=straight edge] (A5) to (A2);
			\draw [style=straight edge] (A9) to (A6);
			\draw [style=straight edge] (A13) to (A10);	
			\end{pgfonlayer}
			\end{tikzpicture}
		\end{center}
		\caption{Extremal examples for $k=4$ and $\ell=5$.}\label{fig:extremal example}
	\end{figure}

\section{Proof of Theorem~\ref{th:main}}\label{sc:maintheorem}
 In this section, we determine $\alpha(k,\ell)$ exactly.
 
\begin{lemma}\label{lm:upper} For $k, \ell\geq 1$,
\[ \alpha(k,\ell) \le  (k-1)(\ell-1)+2.\]
\end{lemma}

\begin{proof}
	The statement is obviously true when $\min(k,\ell)=1$, so assume that $\min(k,\ell)\ge 2$. Without loss of generality $\pi = (1, a_1, a_2, ...., a_{(k-1)(\ell-1)+1})$.
	Consider the sequence $[a_1, a_2, ...,a_{(k-1)(\ell-1)+1}]$. By the Erd\H os-Szekeres theorem, it has either an increasing subsequence of length $k$ or a decreasing subsequence of length $\ell$. If there is an increasing subsequence $[a_{i_1}, a_{i_2},...,a_{i_{k}}]$, then $(1, a_{i_1}, a_{i_2},...,a_{i_{k}})$ would form an increasing cyclic sub-permutation of $\pi$ of length $k+1$. Otherwise, if there is a decreasing subsequence $[a_{i_1}, a_{i_2},...,a_{i_{\ell}}]$, then $(a_{i_1}, a_{i_2},...,a_{i_{\ell}},1)$ would form a decreasing cyclic sub-permutation of $\pi$ of length $\ell+1$.
	\end{proof}
	
\begin{lemma}\label{lm:lower} For $k, \ell\geq 1$,
\[ \alpha(k,\ell) >  (k-1)(\ell-1)+1.\]
In particular, if $\min(k,\ell)\ge 2$,  $\pi=(1,a_1,\ldots,a_{(k-1)(\ell-1)})$ where the sequence $a_i$ is given by one of the formulas in Theorem~\ref{th:lowerbound}
part~\ref{atleast3}~\ref{struct1} or~\ref{struct2}, then $\pi$ does not have an increasing cyclic sub-permutation of length $k+1$ or a decreasing cyclic sub-permutation of length $\ell+1$.
\end{lemma}

\begin{proof}
The lemma is trivial when $\min(k,\ell)=1$.
Assume $\min(k,\ell)\ge 2$ and $\pi = (1,a_1 \ldots, a_{(k-1)(\ell-1)})$, where 
$[a_1,\ldots,a_{(k-1)(\ell-1)}]$ is given by  Theorem~\ref{th:lowerbound}
part~\ref{atleast3}~\ref{struct1}, i.e. for each $(i,j)\in[\ell-1]\times[k-1]$  $a_{(j-1)(\ell-1)+i}=(\ell-1-i)(k-1)+j+1$. (The example given in Figure~\ref{fig:extremal example} for $k=4$ and $\ell=5$ is
$\pi = (1,11,8,5,2,12,9,6,3,13,10,7,4)$.) The other case can be handled analogously.	
	
\noindent We claim $\pi$ does not have an increasing cyclic sub-permutation of length $k+1$ nor does it have a cyclic sub-permutation of length $\ell+1$.
Starting from $a_1$, we can partition the sequence $A = [a_1,\ldots, a_{(k-1)(\ell-1)}]$ into $(k-1)$ decreasing sub-sequences $D_1,\ldots,D_{k-1}$, 
	each consisting of $(\ell-1)$  consecutive elements of the original sequence. In particular, $D_i = [a_{t}, a_{t+1}, \ldots, a_{t+\ell-2}]$ where $t = (i-1)(\ell-1)+1$. In Figure~\ref{fig:extremal example}, this partition corresponds to $[11,8,5,2]$, $[12,9,6,3]$,$[13,10,7,4]$. 
	Let $L$ be the longest increasing cyclic sub-permutation  of $\pi$. Suppose $L = (a_{i_1}, a_{i_2}, \ldots, a_{i_t})$ where 
	$a_{i_1}< a_{i_2}< \ldots< a_{i_t}$. $L$ and $D_i$ has at most $2$ common elements for each $i$, as the elements in $D_i$ are decreasing in $A$.
	If $a_{i_1} = 1$, then $L$ can contain at most one element from each of the $D_i$s. Since there are at most $k-1$ $D_i$s, it follows that $L$ has length at most $k$. If $a_{i_1} \neq 1$, then $a_{i_1} \in D_{j}$ for some $j \in [k-1]$. In this case,  $1\notin L$. Furthermore, $L$ can have at most 2 elements from $D_j$, and at most one element from $D_i$ for each $i \in [k-1]\backslash \{j\}$. Thus $ L$ has length at most $k$.
	
	We can also partition $A$ into $(\ell-1)$ increasing subsequences $C_1,\ldots,C_{\ell-1}$ of length  $(k-1)$. In particular, let $C_i = [c_i, c_i+1,\ldots, c_i+k-2]$ where $c_i= 2+(i-1)(k-1)$. In the example above, $C_1, C_2, C_3, C_4$ would correspond to $[2,3,4]$, $[5,6,7]$,$[8,9,10]$,$[11,12,13]$. Similar to the analysis above,  let $L$ be the longest decreasing cyclic sub-permutation  of $\pi$. Suppose $L = (a_{i_1}, a_{i_2}, \ldots, a_{i_t})$ where $a_{i_1} > a_{i_2}> \ldots> a_{i_t}$. As before, $L$ can have at most $2$ common elements with each $C_i$.
	If $a_{i_t} = 1$, then $L$ can contain at most one element from each of the $C_i$s. Since there are at most $\ell-1$ $C_i$s, it follows that $L$ has length at most $\ell$. If $a_{i_t} \neq 1$, observe that if for some $j$ $L$ has $2$ common elements with $C_j$, then every other $C_i$ ($i\neq j$) can contain at most one element from $L$ since numbers in $C_t$ are strictly larger than all numbers in $C_s$ for $s<t$. Thus $ L$ has length at most $\ell$.
	\end{proof}

Theorem~\ref{th:main} follows from Lemma~\ref{lm:upper} and~\ref{lm:lower}.

\section{The structure of the extremal examples in the linear Erd\H{o}s-Szekeres problem}\label{sc:linear}

We will first consider the linear problem, i.e. sub-permutations will be linear sub-permutations. We will emphasize this by using the vector notation
$\vec{a}=[a_1,\ldots,a_n]$ when talking about linear permutations. Recall the definition of $\gamma_{\vec{a}}, R_{\vec{a}}, V_{\vec{a}}$ in Definition \ref{dfn:grid-function} and \ref{dfn:grid-pair}. It is easy to see that $\gamma_{\vec{a}}$ is an injective (and therefore bijective) function, since for $t_1< t_2$ we have $a_{t_1}\ne a_{t_2}$ and either every increasing sequence ending at $a_{t_1}$ can be extended to an increasing sequence ending at $a_{t_2}$, or every decreasing sequence. The following are immediate from the definitions and prior statements in the lemma:
\begin{lemma}\label{lm:injection} Let $\vec{a}\in\mathbb{S}_{k,\ell}$. The following are true:
\begin{enumerate}[label=\rm{(\arabic*)}]
\item\label{helper} Let $t_1,t_2\in[k\ell]$ such that $t_1<t_2$, and define $i_1,i_2,j_1,j_2$ by $\gamma_{\vec{a}}(t_q)=(i_q,j_q)$ for $q\in [2]$. If $a_{t_1}<a_{t_2}$ then $j_1<j_2$ and if
$a_{t_1}>a_{t_2}$ then $i_1<i_2$.
\item\label{rankingorder} Let $i_2\le i_1$, $j_2\le j_1$ and $\gamma_{\vec{a}}(t_q)=(i_q,j_q)$ where $q\in [2]$. Then $t_2\le t_1$.
\item\label{ranking} $R_{\vec{a}}\in\mathbb{Y}_{\ell,k}$.
\item \label{valuationorder}
For any $i\in[\ell],j\in[k]$ the sequence
$[a_{\gamma^{-1}_{\vec{a}}(i,1)},\ldots,a_{\gamma^{-1}_{\vec{a}}(i,k)}]$ is an increasing subsequence of $\vec{a}$ and the sequence
$[a_{\gamma^{-1}_{\vec{a}}(1,j)},\ldots,a_{\gamma^{-1}_{\vec{a}}(\ell,j)}]$ is a decreasing subsequence of $\vec{a}$.
\item\label{valuation} $V_{\vec{a}}\in\mathbb{Y}_{\ell,k}$.
\item\label{phiinject} $\phi:\mathbb{S}_{k,\ell}\rightarrow \mathbb{Y}_{\ell,k}\times\mathbb{Y}_{\ell,k}$ defined by $\phi(\vec{a})=(R_{\vec{a}},V_{\vec{a}})$ is an injective function
\end{enumerate}
\end{lemma}

\begin{proof}
\ref{helper} follows from the fact that if $a_{t_1}<a_{t_2}$ ($a_{t_1}>a_{t_2}$) then any increasing (decreasing) 
subsequence of $\vec{a}$ ending on $a_{t_1}$ can be extended to a longer increasing (decreasing) subsequence ending at $a_{t_2}$.
This in turn implies \ref{rankingorder}, which gives \ref{ranking}.
\ref{rankingorder} implies that for any $i\in[\ell],j\in[k]$  the sequences  $[\gamma^{-1}(i,1),\gamma^{-1}(i,2),\ldots,\gamma^{-1}(i,k)]$ and
$[\gamma^{-1}(1,j),\gamma^{-1}(2,j),\ldots,\gamma^{-1}(\ell,j)]$ are increasing, and this together with \ref{helper} gives \ref{valuationorder}.
\ref{valuation} follows from \ref{valuationorder}. \ref{ranking} and \ref{valuation} gives that $\phi$ is a well-defined function, and it follows from the definitions that
$\phi$ must be injective, so \ref{phiinject} is true.
\end{proof}

The proof of Theorem~\ref{th:linear} is finished by showing that
\begin{lemma} Let $R=(r_{ij}),V=(v_{ij})\in\mathbb{Y}_{\ell,k}$ and define the sequence $\vec{a}=[a_1,\ldots,a_{k\ell}]$
by $a_t=v_{ij}$ if and only if $t=r_{\ell+1-i,j}$. 
Then $\vec{a}\in\mathbb{S}_{k,\ell}$, $R=R_{\vec{a}}$ and $V=V_{\vec{a}}$. Consequently,
the function $\phi$ defined in {\rm Lemma~\ref{lm:injection}} is a bijection.
\end{lemma}

\begin{proof} 
From the fact that the entries of $V$ (and also the entries of $R$) are unique, it follows that $\vec{a}$ is a well-defined permutation of $[k\ell]$. To show,
$\vec{a}\in\mathbb{S}_{k,\ell}$, it is enough to show that $\vec{a}$ does not have an increasing subsequence of length $k+1$ or a decreasing subsequence of length $\ell+1$.
Assume to the contrary that $[a_{t_1}\ldots,a_{t_{k+1}}]$ is an increasing subsequence of length $k+1$ of $\vec{a}$.
For 
each $q\in[k+1]$ define $(i_q,j_q)$ by  $a_{t_q}=v_{i_qj_q}$. By the pigeonhole principle there is a $q_1<q_2$ such that
$j_{q_1}=j_{q_2}$. Since $V\in\mathbb{Y}_{\ell,k}$, $t_{q_1}<t_{q_2}$ and $a_{t_1}<a_{t_2}$, this implies $i_{q_1}<i_{q_2}$,
so $\ell+1-i_{q_1}>\ell+1-i_{q_2}$, which together with $R\in\mathbb{Y}_{k,\ell}$ gives $t_{q_1}>t_{q_2}$, a contradiction.
The statement that $\vec{a}$ does not have a decreasing subsequence of length $\ell+1$ follows similarly, so $\vec{a}\in\mathbb{S}_{k,\ell}$.
Fix an $i\in[\ell]$ and define the sequence $\vec{t}=[t_1,\ldots,t_k]$ by $t_q=r_{i,q}$. Since
$R\in\mathbb{Y}_{\ell,k}$, $\vec{t}$ is an increasing sequence. Moreover, since $a_{t_q}=v_{\ell+1-i,q}$ and
$V\in\mathbb{Y}_{\ell,k}$, $[a_{t_1},\ldots,a_{t_k}]$ is an increasing subsequence of $\vec{a}$. 
Similarly for any $j\in[k]$ define $\vec{w}=[w_1,\ldots,w_{\ell}]$ by $w_q=r_{q,j}$, then
$\vec{w}$ is increasing and $[a_{w_1},\ldots,a_{w_{\ell}}]$ is a decreasing subsequence of $\vec{a}$.
This implies that for each $i\in[\ell]$ and $j\in[k]$,
$\gamma_{\vec{a}}({r_{i,j}})=(i',j')$ where $i'\ge i$ and $j'\ge j$. Since both $\gamma_{\vec{a}}$ and $\gamma$ are bijections from
$[k\ell]$ to $[\ell]\times[k]$,
we get that $\gamma_{\vec{a}}(r_{i,j})=(i,j)$ and so $r_{ij}=\gamma^{-1}_{\vec{a}}(i,j)$. Thus we obtain $R=R_{\vec{a}}$.
Since for $V_{\vec{a}}=(v^{\star}_{ij})$ we have by definition that $v^{\star}_{ij}=a_{\gamma_{\vec{a}}^{-1}(\ell+1-i,j)}=a_{r_{\ell+1-i,j}}=v_{ij}$, 
we obtain $V=V_{\vec{a}}$. So $\phi(\vec{a})=(R,V)$, therefore $\phi$ is surjective, which together with Lemma~\ref{lm:injection} part~\ref{phiinject} gives that $\phi$ is a bijection.
\end{proof}

We remark that similar ideas appear in \cite{youngtableau} to find the longest increasing subsequence of a sequence.
Fix  $k,\ell\ge 1$ and set $n=k\ell$. Note that the above results imply that if we represent the sequence $\vec{a}=[a_1,\ldots,a_n]$ as the set of $n$ points
$(t,a_t)$ and connect two points $(t_1,a_{t_1})$ and $(t_2,a_{t_2})$ precisely when $\gamma_{\vec{a}}(t_1)$ and $\gamma_{\vec{a}}(t_1)$ agree in one of the coordinates and differ by $1$ on the other, then we get a (potentially somewhat distorted) $\ell\times k$  grid where the slope
of the line from $t_1$ to $t_2$ is positive exactly when $\gamma_{\vec{a}}(t_2)$ agrees with $\gamma_{\vec{a}}(t_1)$ on the first coordinate, and negative otherwise.
The grid may be distorted in the sense that it is formed by quadrangles that are not necessarily rectangles and are not necessarily isomorphic, and the
grid ``balances on one of its corners"; in fact it balances on the grid-point indexed $(\ell+1,1)$ with sequence value $1$. 
Indeed, any sequence $[a_1,\ldots,a_n]$ that is a permutation of $[n]$ is in $\mathbb{S}_{k,\ell}$ precisely when such a grid can be fit on its
$n$-point representation in the plane (where the corner on which the distorted grid balances is the grid-point $(\ell+1,1)$ and has height  $1$). See Figure~\ref{fig:distorted_grid} for illustration.

	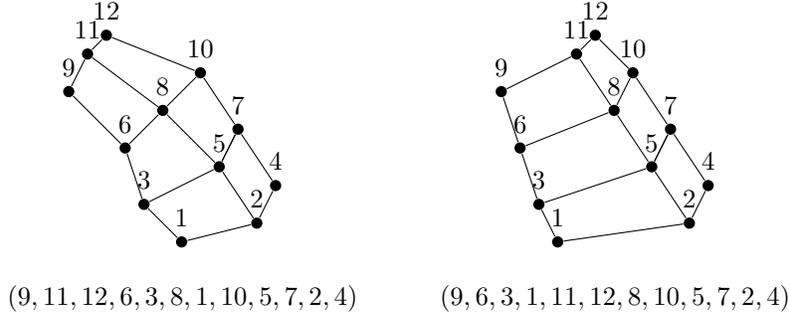
\begin{figure}[hbt]
		\begin{center}
			\tikzstyle{vertex}=[circle,fill=black,inner sep=1.5pt]
			\tikzstyle{straight edge}=[]
			
			\begin{tikzpicture}
			\begin{pgfonlayer}{nodelayer}
			\node [style=vertex,label=above:{$1$}] (A1) at (1.75, 0.25) {};
			\node [style=vertex,label=above:{$2$}] (A2) at (2.75, 0.5) {};
			\node [style=vertex,label=above:{$3$}] (A3) at (1.25, 0.75) {};
			\node [style=vertex,label=above:{$4$}] (A4) at (3, 1) {};
			\node [style=vertex,label=above:{$5$}] (A5) at (2.25, 1.25) {};
			\node [style=vertex,label=above:{$6$}] (A6) at (1., 1.5) {};
			\node [style=vertex,label=above:{$7$}] (A7) at (2.5, 1.75) {};
			\node [style=vertex,label=above:{$8$}] (A8) at (1.5, 2) {};
			\node [style=vertex,label=above:{$9$}] (A9) at (0.25, 2.25) {};
			\node [style=vertex,label=above:{$10$}] (A10) at (2, 2.5) {};
			\node [style=vertex,label=above:{$11$}] (A11) at (0.5, 2.75) {};
			\node [style=vertex,label=above:{$12$}] (A12) at (0.75, 3) {};
			\node at (1.75,-.5) {$(9,11,12,6,3,8,1,10,5,7,2,4)$};
			\node [style=vertex,label=above:{$1$}] (B1) at (6.75, 0.25) {};
			\node [style=vertex,label=above:{$2$}] (B2) at (8.5, 0.5) {};
			\node [style=vertex,label=above:{$3$}] (B3) at (6.5, 0.75) {};
			\node [style=vertex,label=above:{$4$}] (B4) at (8.75, 1) {};
			\node [style=vertex,label=above:{$5$}] (B5) at (8, 1.25) {};
			\node [style=vertex,label=above:{$6$}] (B6) at (6.25, 1.5) {};
			\node [style=vertex,label=above:{$7$}] (B7) at (8.25, 1.75) {};
			\node [style=vertex,label=above:{$8$}] (B8) at (7.5, 2) {};
			\node [style=vertex,label=above:{$9$}] (B9) at (6., 2.25) {};
			\node [style=vertex,label=above:{$10$}] (B10) at (7.75, 2.5) {};
			\node [style=vertex,label=above:{$11$}] (B11) at (7, 2.75) {};
			\node [style=vertex,label=above:{$12$}] (B12) at (7.25, 3) {};
			\node at  (7.5,-0.5) {$(9,6,3,1,11,12,8,10,5,7,2,4)$};
			
			\end{pgfonlayer}
			\begin{pgfonlayer}{edgelayer}
			\draw [style=straight edge] (A9) to (A11);
			\draw [style=straight edge] (A11) to (A12);
			\draw [style=straight edge] (A6) to (A8);
			\draw [style=straight edge] (A8) to (A10);
			\draw [style=straight edge] (A3) to (A5);
			\draw [style=straight edge] (A5) to (A7);
			\draw [style=straight edge] (A1) to (A2);	
			\draw [style=straight edge] (A2) to (A4);								
			\draw [style=straight edge] (A9) to (A6);
			\draw [style=straight edge] (A6) to (A3);
			\draw [style=straight edge] (A3) to (A1);
			\draw [style=straight edge] (A11) to (A8);
			\draw [style=straight edge] (A8) to (A5);
			\draw [style=straight edge] (A5) to (A2);
			\draw [style=straight edge] (A5) to (A7);
			\draw [style=straight edge] (A12) to (A10);			
			\draw [style=straight edge] (A10) to (A7);
			\draw [style=straight edge] (A7) to (A4);
			\draw [style=straight edge] (B9) to (B11);
			\draw [style=straight edge] (B11) to (B12);
			\draw [style=straight edge] (B6) to (B8);
			\draw [style=straight edge] (B8) to (B10);
			\draw [style=straight edge] (B3) to (B5);
			\draw [style=straight edge] (B5) to (B7);
			\draw [style=straight edge] (B1) to (B2);	
			\draw [style=straight edge] (B2) to (B4);								
			\draw [style=straight edge] (B9) to (B6);
			\draw [style=straight edge] (B6) to (B3);
			\draw [style=straight edge] (B3) to (B1);
			\draw [style=straight edge] (B11) to (B8);
			\draw [style=straight edge] (B8) to (B5);
			\draw [style=straight edge] (B5) to (B2);
			\draw [style=straight edge] (B5) to (B7);
			\draw [style=straight edge] (B12) to (B10);			
			\draw [style=straight edge] (B10) to (B7);
			\draw [style=straight edge] (B7) to (B4);
			\end{pgfonlayer}
			\end{tikzpicture}
		\end{center}
		\caption{Two examples of extremal sequences for the linear Erd\H{o}s-Szekeres theorem for $k=4$ and $\ell=5$ with distorted grid representation.
		They have the same valuation but different ranking.}\label{fig:distorted_grid}
	\end{figure}

\section{The structure of the extremal examples in the circular Erd\H{o}s-Szekeres problem}
\label{sc:lowerbound}

We devote this section to the proof of Theorem~\ref{th:lowerbound}. 
The statement is obvious when $\min(k,\ell)=1$, so we assume that $\min(k,\ell)\ge 2$.
For this case we have shown in Lemma~\ref{lm:lower} that the structures described in Theorem~\ref{th:lowerbound} are all in $\mathbb{C}_{k,\ell}$, the proof of Theorem~\ref{th:lowerbound} is finished by showing that these structures are the only elements od $\mathbb{C}_{k,\ell}$. 
Moreover, since any cyclic permutation of length at least $3$ that is not the increasing (decreasing) permutation contains a decreasing (increasing) sub-permutation of length at least $3$, the statement follows when $\min(k,\ell)=2$. So it is enough to focus on the case when $\min(k,\ell)\ge 3$.

We will define $\mathbb{C}^{\star}_{k,\ell}$ as the set of those sequences in $\mathbb{S}_{k-1,\ell-1}$ that, taken as as cyclic permutations have no increasing cyclic sub-permutation of length 
$k+1$,  and no decreasing cyclic sub-permutations of length $\ell+1$. For the ease of reference, given a sequence $\vec{\rho}\in\mathbb{C}^{\star}_{k,\ell}$ we will use $\rho$ to denote the cyclic permutation corresponding to $\vec{\rho}$.

As an increasing (decreasing) cyclic sub-permutation of a cyclic permutation either starts (ends) with $1$ or does not contain $1$, the following is obvious:
\begin{lemma}\label{lem:equiv} Let $k,\ell\in\mathbb{Z}$ with $\min(k,\ell)\ge 2$.
$(1,a_1,\ldots,a_{(k-1)(\ell-1)})\in\mathbb{C}_{k,\ell}$  
if and only if $[a_1-1,a_2-1,\ldots,a_{(k-1)(\ell-1)}-1]\in\mathbb{C}^{\star}_{k,\ell}$.
\end{lemma}

By the above Lemma, to characterize the extremal examples in the cyclic Erd\H{o}s-Szekeres theorem it is enough to determine 
$\mathbb{C}^{\star}_{k,\ell}$. The proof of Theorem~\ref{th:lowerbound} is concluded by showing that
\begin{lemma} Let $k,\ell\in\mathbb{Z}$ with $\min(k,\ell)\ge 3$ and $\vec{\rho}=[a_1,\ldots,a_{(k-1)(\ell-1)}]\in\mathbb{C}^{\star}_{k,\ell}$. Then we have one of the following:
\begin{enumerate}[label=\rm{(\roman*)}]
\item\label{first} For each $i\in[\ell-1]$ and $j\in[k-1]$ $a_{(j-1)(\ell-1)+i}=(\ell-1-i)(k-1)+j$.
\item\label{second} For $i\in[\ell-1]$ and $j\in[k-1]$ $a_{(i-1)(k-1)+j}=(j-1)(\ell-1)+(\ell-i)$.
\end{enumerate}
\end{lemma}

\begin{proof}
Let $\vec{\rho}=[a_1,\ldots,a_{(k-1)(\ell-1)}]\in\mathbb{C}^{\star}_{k,\ell}\subseteq\mathbb{S}_{k-1,\ell-1}$. 
For shortness, we will use $\gamma$ for $\gamma_{\vec{\rho}}$.
For each $i\in[\ell-1]$, define the sequences $C_i=[c_{i,1},\ldots,c_{i,k-1}]$ by
$c_{i,j}=a_{\gamma^{-1}(i,j)}$ and  for each $j\in[k-1]$, let $D_j=[c_{1,j},c_{2,j},\ldots,c_{\ell-1,j}]$. 
Clearly, $C_1,\ldots,C_{\ell-1}$ and $D_1,\ldots,D_{k-1}$ partition the elements
of $\vec{\rho}$.
By Lemma~\ref{lm:injection} part~\ref{valuationorder} the $C_i$s are increasing and the $D_j$s are decreasing subsequences of $\vec{\rho}$.
As $\vec{\rho}\in\mathbb{C}^{\star}_{k,\ell}$, the cyclic permutation $\rho$ does not have an
 increasing cyclic sub-permutation of length $k+1$ or decreasing cyclic sub-permutation of length $\ell+1$.
We have two possibilities
\begin{description}
	\item Case 1: $\gamma^{-1}(\ell-1,1)< \gamma^{-1}(1,k-1).$

			As for each $j\in[k-1]$, $D_j$ is an decreasing subsequence of $\vec{\rho}$ we get
\begin{eqnarray*}
a_{\gamma^{-1}(1,j)} > a_{\gamma^{-1}(2,j)}>\cdots>a_{\gamma^{-1}(\ell-1,j)} \end{eqnarray*}
Using this for $j\in\{1,k-1\}$ and the condition, for each $i\in[\ell-2]$ we have
		  $$(c_{1,k-1}, c_{2,k-1}, \ldots, c_{\ell-i,k-1}, c_{\ell-i-1,1}, c_{\ell-i,1},\cdots, c_{k-1,1})$$ 
		  is a cyclic sub-permutation of length $\ell+1$ of the cyclic permutation $\rho$.
		  Since this can not be an decreasing sub-permutation, we must have  $c_{\ell-i, k-1} <  c_{\ell-i-1,1}$. 
		  Let $i^{\star}\in[\ell-i-1]$ and $j\in[k-1]$.  
		  As $D_{1}$ is decreasing and $C_{i^{\star}}$ is increasing, we have
		  $c_{\ell-i,k-1}<c_{\ell-i-1,1}\le c_{i^{\star},1} \le c_{i^{\star},j}$ and consequently
		  $c_{\ell-i,k-1}\le (k-1)i$. Using that $C_{\ell-i}$ is increasing, induction on $i$ gives that
		  $c_{\ell-i,j}=a_{\gamma^{-1}(\ell-i,j)}=(i-1)(k-1)+j$.

	Since $C_1$ and $C_{\ell-1}$ are both increasing subsequences of $\vec{\rho}$ and
	$C_{\ell-1}$ contains the smallest $(k-1)$ elements of $[(k-1)(\ell-1)]$, we must have that for each
	$j\in[k-2]$ that $\gamma^{-1}(1,j+1)>\gamma^{-1}(\ell-1,j)$, 
	otherwise
	$$(c_{\ell-1,j},c_{\ell-1,j+1},\ldots,c_{\ell-1,k-1},c_{1,1},c_{1,2},\ldots, c_{1,j+1})$$
	would form an increasing cyclic sub-permutation of length $k+1$ of $\rho$. 
	Using the fact that $D_j$ is a sub-permutation and induction on $j$,
	for each $j\in[k-1]$ we get $\gamma^{-1}(\ell-i,j)=(j-1)(\ell-1)+\ell-i$.
	
	Combining these we must have that for $i\in[\ell-1]$ and $j\in[k-1]$ $a_{(j-1)(\ell-1)+i}=(\ell-1-i)(k-1)+j$, giving case~\ref{first} of this lemma.
		   
	\item Case 2: $\gamma^{-1}(l-1,1) > \gamma^{-1}(1,k-1).$ 
			
	As before, we get that for each $j\in[k-2]$ the sequence
	$$(c_{l-1,1}, c_{l-1,2}, \ldots, c_{l-1,k-j}, c_{1,k-j-1}, c_{1,k-j},\ldots, c_{1,k-1})$$
	 is a cyclic sub-permutation of length $k+1$ of $\rho$, and as it can not be increasing, we have
	 $c_{l-1, k-j} >  c_{1,k-j-1}$. 
	 Using the same logic as in Case 1 we obtain for each $j\in[k-1]$ and $i\in[\ell-1]$
	 $a_{\gamma^{-1}(i,j)}=(j-1)(\ell-1)+(\ell-i)$.
	 
	Again, for each $i\in[\ell-1]$ we have $\gamma^{-1}(i+1,1) > \gamma^{-1}(i,k-1)$, otherwise 
	$$(c_{i,k-1},c_{i+1,k-1},\ldots,c_{\ell-1,k-1},c_{1,1}, c_{2,1},\ldots,c_{i+1,1})$$
	 forms decreasing cyclic sub-permutation of length $\ell+1$ of $\rho$.
	 We obtain that $\gamma^{-1}(i,j)=(i-1)(k-1)+j$.
	  Combining these we must have that for $i\in[\ell-1]$ and $j\in[k-1]$ $a_{(i-1)(\ell-1)+j}=(j-1)(\ell-1)+(\ell-i)$,  
	 giving case~\ref{second} of this lemma.

\end{description}

\end{proof}

For $k,\ell\ge 2$, set $n=(k-1)(\ell-1)$. and consider the sequence $\vec{\rho}=[1,a_1,\ldots,a_n]$; i.e. use the sequence representation of
the cyclic permutation or $\rho$ that starts with $1$.
It is worth noting that $\rho\in\mathbb{C}_{k,\ell}$ precisely when taking the $n+1$ points representing $\vec{\rho}$ in the plane and
putting in the grid lines corresponding to $[a_1-1,\ldots,a_{n}-1]$ described in the end of the previous section to the $n$ points of the form $(i,a_i)$, they form a non-distorted grid, i.e. a grid with rectangles (and not just quadrangles) that are of the same size (in fact, the ratio of the side length of each rectangle is $\frac{k-1}{\ell-1}$), and the point $(1,1)$ lies on either the first or the last line with positive slope, as in Figure~\ref{fig:extremal example}.

\section{Acknowledgement}

We would like to thank  L\'{a}szl\'{o} Sz\'{e}kely for proposing this problem to us and for his many suggestions on this paper. We would also like to thank Joshua Cooper for his valuable input on the topic. We are indebted to the anonymous referees who suggested a number of improvements to this paper.

\end{document}